\documentclass[11pt,a4paper]{article}
\usepackage{amssymb}
\usepackage{amsmath}
\usepackage{amsthm, hyperref}
\usepackage{authblk}

\newtheorem{thr}{Theorem}[section]
\newtheorem{lem}[thr]{Lemma}

\newtheorem{conj}[thr]{Conjecture}
\theoremstyle{definition}

\newtheorem{cor}[thr]{Corollary}
\newtheorem{cla}[]{Claim}
\newcommand{\m}{\mathrm{min}}
\newcommand{\cM}{\mathcal{M}}

\DeclareMathOperator{\cm}{cm}

\title{Connected matching in graphs with independence number two}

\author[1]{Rong Chen\footnote{Email: rongchen@fzu.edu.cn (R. Chen).} }
\author[1]{Zijian Deng\footnote{Email: zj1329205716@163.com.(corresponding author).} }

\affil{Center for Discrete Mathematics,
Fuzhou University, China}

\date{}

\usepackage{indentfirst}
\begin{document}

\maketitle

\begin{abstract}
	
A matching $M$ in a graph $G$ is {\em connected} if $G$ has an edge linking each pair of edges in $M$. The problem to find large connected matchings in graphs $G$ with $\alpha(G)=2$ is closely related to Hadwiger's conjecture for graphs with independence number 2. 
The problem of finding a large connected matching in a general graph is NP-hard.
F{\"u}redi et al. in 2005 conjectured that each $(4t-1)$-vertex graph $G$ with $\alpha(G)=2$ contains a connected matching of size at least $t$. 
Cambie recently showed that if this conjecture is false, then so is Hadwiger's conjecture. In this paper, we present a number of properties possessed by a counterexample to F{\"u}redi et al.'s conjecture, and then using these properties, we prove that F{\"u}redi et al.'s conjecture holds for $t\leq22$.


\end{abstract}

\noindent\textbf{Mathematics Subject Classification}: 05C70; 05C83

\noindent\textbf{Keywords}: Hadwiger's conjecture; connected matchings; matchings

\section{Introduction}\label{sec:intro}

	All graphs considered in this paper are finite and simple. Let $K_{n}$ be a clique on $n$ vertices.
Let $\omega(G), \alpha(G)$ and $\chi(G)$ be the clique number, the independence number and the chromatic number of a graph $G$, respectively. If a graph $H$ can be obtained from $G$ by deleting vertices or edges and contracting edges, then $H$ is called a \emph{minor} of $G$. Given disjoint subsets $A$ and $B$ of $V(G)$, we say that $A$ is \emph{complete} to $B$ if  each vertex in $A$ is adjacent to all vertices in $B$. Analogously, $A$ is \emph{anti-complete} to $B$ if there exists no edge between $A$ and $B$.


In 1943, Hadwiger \cite{HU} proposed the following famous conjecture.

\begin{conj}[\cite{HU}]\label{HC}
Every graph $G$ contains $K_{\chi(G)}$ as a minor.
\end{conj}



When we consider those graphs with independence number at most 2, Hadwiger's conjecture implies the following conjecture.

\begin{conj}\label{weakHadw}
Every graph $G$ with $\alpha(G)\leq2$ contains $K_{\lceil \frac{|V(G)|}{2}\rceil}$ as a minor.
\end{conj}

Plummer, Stiebitz, and Toft \cite{MM} showed that Conjecture \ref{weakHadw} is equivalent to Hadwiger's conjecture for graphs with independence number at most 2.
Conjecture \ref{weakHadw} is still unsolved and it might be a possible direction to disprove Hadwiger's conjecture.

Let $\cM$ be a set of vertex-disjoint subgraphs of a graph $G$. For each pair of subgraphs $M, M'\in\cM$, if $G$ has an edge linking $M$ and $M'$, then we say that $\cM$ is {\em connected}. If $\cM$ is connected and consists of edges, 
then we say that $\cM$  is a \emph{connected matching}. Let $\cm(G)$ represent the size of a largest connected matching in $G$.

The problem to find large connected matchings in graphs $G$ with $\alpha(G)=2$ is closely related to Conjecture \ref{weakHadw}. The problem of finding a large connected matching in a general graph is NP-hard.
F{\"u}redi, Gy{\'a}rf{\'a}s and Simonyi~\cite{ZA} in 2005 proposed a conjecture about the following precise version on the question of determining $\cm(G)$.
\begin{conj}[\cite{ZA}]\label{conjFGS}
	Let $G$ be a $(4t-1)$-vertex graph with $\alpha(G) =2$. 
Then $\cm(G) \ge t.$
\end{conj}

Cambie \cite{SC} in 2021 showed that if Conjecture~{\rm\ref{HC}} is true, then so is Conjecture~{\rm\ref{conjFGS}}.
Conjecture~{\rm\ref{conjFGS}} would be sharp due to the graph $K_{2t-1} \cup K_{2t-1}.$
The current best known lowerbound for $\cm(G)$ in this setting is of the order $\Theta( n^{4/5} \log^{1/5} n )$ by~\cite{JF}.
In this paper, we examine structural properties possessed by counterexamples to Conjecture \ref{conjFGS}. Theorems \ref{thm1} and \ref{lem2} are two main results obtained in this paper.
\begin{thr}\label{thm1}
Let $G$ be a minimal counterexample to Conjecture {\rm\ref{conjFGS}}.
For any disjoint subsets $S_1, S_2\subseteq V(G)$ {\rm(}some $S_i$ maybe empty{\rm)}, if $G[S_1]$ and $G[S_2]$ are anti-complete cliques, then $|S_1|+|S_2|\leq t-4$.
\end{thr}
\begin{thr}\label{lem2}
Let $G$ be a $(4t-1)$-vertex graph with $\alpha(G)=2$ and with $\cm(G)\le t-1$. 
Let $S_0$, $S_1$ and $S_2$ be disjoint subsets in $V(G)$ {\rm(}some $S_i$ maybe empty{\rm)}, all of which induce cliques in $G$. Assume that $S_0$ is complete to  $S_1\cup S_2$, and $S_1$ is anti-complete to $S_2$. Then $|S_0|+|S_1|+|S_2|\leq t-1$.
\end{thr}

For any non-empty disjoint subsets $S_1, S_2\subseteq V(G)$, if $\alpha(G)=2$ and 
$S_1$ is anti-complete to $S_2$, then $G[S_1]$ and $G[S_2]$ are obviously cliques. 
The reason why we ask $G[S_1]$ and $G[S_2]$ in Theorems \ref{thm1} and \ref{lem2} being clique is that some $S_i$ maybe empty. Theorem \ref{thm1} implies $\omega(G)\leq t-4$. Moreover, using Theorem \ref{thm1} and following the idea in \cite{ZA}, we can  prove 

\begin{thr}\label{thm3}
Conjecture \ref{conjFGS} holds for $t\leq22$.
\end{thr}
\begin{proof}[ Proof of Theorem \ref{thm3} assuming Theorem \ref{thm1}.]
Assume to the contrary that $G$ is a minimal counterexample to Theorem \ref{thm3}.
Arbitrarily choose a vertex $v\in V(G)$. Set 
$Y:=V(G)\backslash (N(v)\cup\{v\})$.
By Theorem \ref{thm1}, we have $|Y|\leq t-5$, implying $|N(v)|\geq 3t+3$. Then some vertex $u\in N(v)$ is non-adjacent to at most $\frac{|Y|(t-6)}{|N(v)|}$ vertices in $Y$ by  Theorem \ref{thm1} again. That is, the vertex $u$ is non-adjacent to at most 3 vertices in $Y$ as $t\leq22$. This allows us to carry out the inductive proof. Removing $v$, $u$ and two further vertices in $Y$ (as many from $Y-N(u)$ as possible) the remaining graph has a connected $(t-1)$-matching and the edge $uv$ extends it to a connected $t$-matching, which is a contradiction.
\end{proof}



\section{Properties Possessed by a Counterexample to Conjecture {\rm\ref{conjFGS}}}
For any vertex subset $A$ of a graph $G$, let $G[A]$ denote the subgraph of $G$ induced by $A$. 
We denote by $N_{G}(A)$ the vertex set in $V(G)\backslash A$ that has a neighbour in $A$.
Set $N_{G}[A]:=N_{G}(A)\cup A$
and $\overline{N_G(A)}:=V(G)\backslash N_{G}[A]$. Note that $\overline{N_G(A)}$ induces a clique when $\alpha(G)=2$.
When there is no danger of confusion, all subscripts will be omitted.
For simplicity, when $A = \{a\}$, the sets $N(\{a\})$, $N[\{a\}]$, and $\overline{N(\{a\})}$  are denoted by $N(a)$, $N[a]$, and $\overline{N(a)}$, respectively.

Let $\cM$ be a set of vertex-disjoint subgraphs of a graph $G$. If $\cM$ is connected and consists of vertices and edges, 
then we say that $\cM$  is a \emph{generalized connected matching}. Evidently, when $\cM$  is a generalized connected matching, the union of subgraphs in $\cM$ that have only one vertex induces a clique of $G$, and contracting each edge in $\cM$ gets a $|\cM|$-clique minor.

In the rest of this section, we always assume that $G$ is a graph of order $4t-1$ with $\alpha(G)=2$.
\begin{lem}[\cite{SC}]\label{k1+k2<m}
When $\cm(G)\le t-1$, every generalized connected matching of $G$ has size at most $\cm(G)$.
\end{lem}

By Lemma \ref{k1+k2<m} and the definition of generalized connected matching, for any counterexample $G$ to Conjecture \ref{conjFGS}, we have $\omega(G)\leq t-1$. 



\begin{proof}[Proof of Theorem \ref{lem2}.]
Set $s_i:=|S_i|$ for each $i\in\{0,1,2\}$. 
Since Lemma \ref{k1+k2<m} implies $\omega(G)\leq t-1$, we have that $$s_0+s_1,\ s_0+s_2\leq t-1,\  \text{so}\ s_0+s_1+s_2\leq 2t-2 \eqno{(2.1)}$$ as $G[S_0\cup S_1]$ and $G[S_0\cup S_2]$ are cliques.
Suppose to the contrary that $s_0+s_1+s_2 \geq t$. Then $S_1, S_2\neq \emptyset$ by (2.1). 
Set
\begin{eqnarray*}
  C_1 &:=& \{x\in V(G)\backslash (S_0\cup S_1 \cup S_2):\ x\ \text{has a non-neighbour in}\ S_2\},\\
  C_2 &:=& \{x\in V(G)\backslash (S_0\cup S_1 \cup S_2):\ x\ \text{has a non-neighbour in}\ S_1\}, \\
  C_0 &:=& V(G)\backslash (S_0\cup S_1 \cup S_2 \cup C_1 \cup C_2).
\end{eqnarray*}
Then {\bf (a)} $C_0$ is complete to $S_1 \cup S_2$ by definition. Since $\alpha(G)=2$, $S_1, S_2\neq \emptyset$, and $S_1$ is anti-complete to $S_2$, {\bf (b)} we have that $C_1$ is complete to $S_1$ and $C_2$ is complete to $S_2$, implying $C_1 \cap C_2=\emptyset$. That is, $(S_0, S_1, S_2, C_0, C_1, C_2)$ is a partition of $V(G)$.

By (2.1), we have  $|C_1|+|C_2|+|C_0|\geq 2t+1$, so either $|C_1\cup C_0|\geq t+s_2-s_1$ or $|C_2\cup C_0|\geq t+s_1-s_2$. By symmetry we may assume that $|C_1\cup C_0|\geq t+s_2-s_1$. Then  $|C_1\cup C_0|>s_2$ as $s_0+s_1\leq t-1$ by (2.1). Let $H$ be the bipartite subgraph induced by the edges between $C_1\cup C_0$ and $S_2$. Let $M$ be a matching of $H$ with $|M|$ maximal. Since each edge in $M$ has exactly one end in $S_2$ and $G[S_2]$ is a clique, $M$ is connected. (This fact will be frequently used in this section without reference.) When $|M|=s_2$, the set $M\cup S_1 \cup S_0$ is a generalized connected matching of size at least $t$ by (a) and (b), a contradiction to Lemma \ref{k1+k2<m}. So $|M|<s_2$. Since $|C_1\cup C_0|>s_2$, by K\"{o}nig's theorem, there exists a set $R \subseteq S_2\cup C_1\cup C_0$ with $|R|< s_2$ meeting all edges of $H$.
Since $S_2 \backslash R$ is anti-complete to $(C_1\cup C_0) \backslash R$, the graph  $G[(C_1\cup C_0) \backslash R]$ is a clique of size at least $t-s_1+1$ as $|C_1\cup C_0|\geq t+s_2-s_1$, implying that $G[S_1 \cup (C_1\cup C_0) \backslash R]$ contains a $t$-clique, a contradiction to Lemma \ref{k1+k2<m}.
\end{proof}

Thus, for any counterexample $G$ to Conjecture {\rm\ref{conjFGS}}, and any non-empty vertex-disjoint sets $S_1,S_2\subseteq V(G)$ such that $S_1$ is anti-complete to $S_2$, since $\alpha(G)=2$ implies that  $G[S_1]$ and $G[S_2]$ are cliques, we have $|S_1|+|S_2|\leq t-1$ by Theorem \ref{lem2}. In particular, for any $v\in V(G)$, we have $|\overline{N(v)}|\leq t-2$, which was proved in  \cite{ZA}.

\begin{lem}\label{lem-triangle}
Assume that $\cm(G)\le t-1$. For any triangle $v_1v_2v_3v_1$ in $G$, we have $|\bigcap\limits_{i=1}^3 N(v_i)|\geq t+2$.
\end{lem}
\begin{proof}
By Theorem \ref{lem2}, we have that $|\overline{N(v)}|\leq t-2$ for any $v\in V(G)$.
Set $N:=\bigcap\limits_{i=1}^3 N(v_i)$ and $M:=V(G)\backslash (\bigcup\limits_{i=1}^3 v_i \cup N)$.
Since $M\subseteq \bigcup\limits_{i=1}^3 \overline{N(v_i)}$, we have $|M|\leq |\bigcup\limits_{i=1}^3 \overline{N(v_i)}|\leq 3(t-2)$. Thus, $|N|\geq 4t-1-(3t-6)-3=t+2$.
\end{proof}

\begin{lem}\label{lem-matching}
Assume that $\cm(G)\le t-1$. For any disjoint subsets $A, B \subseteq V(G)$ with $|A|\leq |B|=t-1$, the bipartite subgraph $H$ of $G$ induced by the edges between $A$ and $B$ contains a matching of size $|A|$.
\end{lem}
\begin{proof}
Suppose not. By K\"{o}nig's theorem, there exists a set $R \subseteq A\cup B$ with $|R|\leq |A|-1$ meeting all edges of $H$, so $G[A \backslash R]$ and $G[B \backslash R]$ are anti-complete cliques as $\alpha(G)=2$, which contradicts to Theorem \ref{lem2} as $|A \backslash R|+|B \backslash R|\geq t$.
\end{proof}

\begin{thr}[\cite{ZA}]\label{17}
Conjecture {\rm\ref{conjFGS}} holds for $1\leq t\leq 17$.
\end{thr}

\begin{lem}\label{lem-3sets}
Assume that $\cm(G)\le t-1$. Let $(S_1, S_2, S_3)$ be a partition of $V(G)$  {\rm(}some $S_i$ maybe empty{\rm)} such that $G[S_1]$ and $G[S_2]$ are anti-complete cliques, and such that 
$S_2$ is complete to $S_3$. Then $|S_1| + |S_2| \leq t - 4$.
\end{lem}

\begin{proof}
By Theorem \ref{17}, we may assume that $t \geq 18$.
Set $s_i:=|S_i|$ for each integer $1\leq i\leq 3$. Suppose to the contrary that $s_1 + s_2 \geq t - 3$.
Since Theorem \ref{lem2} implies that $s_1 + s_2 \leq t - 1$, we have $s_3\geq 3t \geq 54$, so $G[S_3]$ contains a triangle $T$ by Ramsey theorem.
Since $G[S_2 \cup V(T)]$ is a clique and $\omega(G)\leq t-1$ by Lemma \ref{k1+k2<m}, we have $s_2\leq t-4$, so $S_1\neq \emptyset$.

\begin{cla}\label{claim4}
For any triangle $v_1v_2v_3v_1$ in $G[S_3]$, we have $|\bigcap\limits_{i=1}^3 N_{G[S_3]}(v_i)|\leq t-2$.
\end{cla}
\begin{proof}
Assume to the contrary that $|\bigcap\limits_{i=1}^3 N_{G[S_3]}(v_i)|\geq t-1$. Since $s_3\geq 3t$, by Lemma \ref{lem-matching}, the bipartite subgraph of $G$ induced by the edges between $S_1$ and $\bigcap\limits_{i=1}^3 N_{G[S_3]}(v_i)$ contains a matching $M$ of $G$ with size $s_1$. Then $M\cup S_2 \cup \{v_1,v_2,v_3\}$ is a generalized connected matching of size at least $t$ as $s_1 + s_2 \geq t - 3$, a contradiction to Lemma \ref{k1+k2<m}.
\end{proof}

Let $xyzx$ be a triangle of $G[S_3]$. Set
\begin{eqnarray*}
S''_1:= &\{v\in S_1:\ v\in N(x)\cap N(y)\cap N(z)\},\\
A:=&\{v\in S_3:\ v\in N(x)\cap N(y)\cap N(z)\}.
\end{eqnarray*}
Set $S'_1:=S_1\backslash S''_1$ and $a:=|S'_1|$. Then $|S''_1|=s_1-a$.
By Lemma \ref{lem-triangle}, $|A|+(s_1-a)+s_2\geq t+2$. Moreover, since $s_1+s_2\leq t-1$, we have $|A|\geq a+3$. Combined with Claim \ref{claim4}, we have {\bf (a)} $ a+3\leq |A|\leq t-2$.

Let $H_1$ be the bipartite subgraph of $G$ induced by the edges between $S'_1$ and $A$.
We claim that $H_1$ contains no matching of size $a$. Suppose to the contrary that $M_1$ is a matching of $H_1$ with size $a$. Since $$|S_3 \backslash (V(M_1) \cup \{x,y,z\})|\geq 3t-(a+3)>t-1,$$ the bipartite subgraph of $G$ induced by the edges between $S_3\backslash (V(M_1) \cup \{x,y,z\})$ and $S''_1$ contains a matching $M_2$ of size $s_1-a$ by Lemma \ref{lem-matching}. Then $M_1\cup M_2\cup\{x,y,z\} \cup S_2$ is a generalized connected matching of size at least $t$, a contradiction to Lemma \ref{k1+k2<m}. Thus, the claim holds.

By (a), K\"{o}nig's theorem and the claim proved in the last paragraph,
there is  a set $R \subseteq V(H_1)$ with $|R|\leq a-1$ meeting all edges of $H_1$. Without loss of generality we may further assume that $R$ is chosen with $|R|$ as small as possible. Set $B:= S'_1 \cap R$ and $D: =A \cap R$. Then $S'_1\backslash B$ is anti-complete to $A \backslash D$. 
Since $|R|\leq a-1$ and $a=|S'_1|$, we have {\bf (b)} $S'_1\backslash B$ is non-empty, so $G[A \backslash D]$ is a clique as  $\alpha(G)=2$. Moreover, since $|A|\geq a+3$ by (a), we have {\bf (c)} $|B|<|A \backslash D|\geq 4$.

Let $H_2$ be the bipartite subgraph of $G$ induced by the edges between $B$ and $A \backslash D$.
We claim that $H_2$ has a matching $M$ of size $|B|$. Assume not, implying that $B$ is non-empty. 
Then by (c) and K\"{o}nig's theorem again, there exists a set $R' \subseteq B\cup (A\backslash D)$ with $|R'|\leq |B|-1$ meeting all edges of $H_2$. Moreover, since $R$ meets all edges of $H_1$, so is $D \cup R'$, a contradiction to the choice of $R$ as $|D \cup R'| < |R|$. Thus, the claim holds. 

Set $$F:=\{v\in S_3\backslash (A\cup\{x,y,z\}):\ v\ \text{has a non-neighbour in}\ S'_1\backslash B\}.$$ 
Since $S'_1\backslash B$ is non-empty and anti-complete to $A \backslash D$ by (b), the graph $G[A \backslash D]$ is a clique and complete to $F$. Since $G[A \backslash D]$ contains a triangle by (c), it follows from Claim \ref{claim4} that $|F|\leq t-5$ as $A$ is also complete to $\{x,y,z\}$. 
Let $L$ be the subset of $S_3\backslash (A\cup\{x,y,z\})$ consisting of vertices that are complete to $S'_1\backslash B$. Then $(L, A, F, \{x,y,z\})$ is a partition of $S_3$. Since $s_1+s_2\leq t-1$, $|F|\leq t-5$ and $|A|\leq t-2$ by (a), we have $|L|\geq t+4$. Hence, by Lemma \ref{lem-matching}, the bipartite subgraph induced by the edges between $L$ and $(S_1\backslash B)\cup \{x,y,z\}$ contains a matching $M'$ of size $|S_1|-|B|+3$. Then $M\cup M'\cup S_2$ is a generalized connected matching  of size at least $t$, a contradiction to Lemma \ref{k1+k2<m}.
\end{proof}

\begin{cor}\label{omega}
\textit{If $\cm(G)\le t-1$, then $\omega(G)\leq t-4$.}
\end{cor}
\begin{proof}
Let $S_1$ be a $\omega(G)$-clique of $G$. Set $S_2:=\emptyset$ and $S_3:=V(G)-V(S_1)$. Then $\omega(G)\leq t-4$ follows from Lemma \ref{lem-3sets}.
\end{proof}

A connected matching $M$ of a graph $H$ is \emph{dominating} if each edge in $M$ is adjacent to all vertices in $V(H)\backslash V(M)$. 

\begin{lem}\label{lem1}
Let $G$ be a minimal counterexample to Conjecture {\rm\ref{conjFGS}}. Then $G$ contains no dominating matching.
\end{lem}
\begin{proof}
Assume to the contrary that $G$ contains a dominating matching $M$ of size $k$ for some positive integer $k$. Since $M$ is connected and $G$ is a counterexample to Conjecture {\rm\ref{conjFGS}}, we have $1\leq k<t$. Moreover, since $|V(G) \backslash V(M)|>4(t-k)$, by the minimality of $G$, the graph $G\backslash V(M)$ has a connected $(t-k)$-matching $M'$. Then $M\cup M'$ is a connected $t$-matching in $G$ by the definition of dominating matchings, a contradiction.
\end{proof}

Now we prove our main result Theorem \ref{thm1}. For convenience, we restate it here.
\begin{thr}
Let $G$ be a minimal counterexample to Conjecture {\rm\ref{conjFGS}}. 
For any disjoint subsets $S_1, S_2\subseteq V(G)$ {\rm(}some $S_i$ maybe empty{\rm)}, if $G[S_1]$ and $G[S_2]$ are anti-complete cliques, then $|S_1|+|S_2|\leq t-4$.
\end{thr}
\begin{proof}
Set $s_i:=|S_i|$ for each $1\leq i\leq2$. 
Assume to the contrary that $s_1+s_2\geq t-3$. By Corollary \ref{omega}, {\bf (a)} neither $S_1$ nor $S_2$ is empty. Moreover, by Theorem \ref{lem2}, we have $t-1\geq s_1+s_2\geq t-3$.
Without loss of generality, we may further assume that $S_1$ and $S_2$ are chosen with $s_1+s_2$ as maximal as possible.
Set
\begin{eqnarray*}
 C_1&:=& \{x\in V(G)\backslash (S_1\cup S_2):\  x\ \text{has a non-neighbour in}\ S_2\}, \\
 C_2&:=& \{x\in V(G)\backslash (S_1\cup S_2):\  x\ \text{has a non-neighbour in}\ S_1\}, \\
C_0&:=& V(G)\backslash (S_1\cup S_2 \cup C_1 \cup C_2).
\end{eqnarray*}
Then {\bf (b)} $C_0$ is complete to $S_1 \cup S_2$ by definition. Since $\alpha(G)=2$ and $S_1$ is anti-complete to $S_2$, we have that {\bf (c)} $C_1$ is complete to $S_1$ and $C_2$ is complete to $S_2$ by (a). Hence, it follows from Lemma \ref{lem-3sets} that neither $C_1$ nor $C_2$ is empty.


For each $1\leq i\leq2$, let $H_i$ be the bipartite subgraph of $G$ induced by the edges between $C_i$ and $S_{3-i}$. 
\begin{cla}\label{matching}
For each $1\leq i\leq2$, the graph $H_i$ contains a connected $s_{3-i}$-matching. 
\end{cla}
\begin{proof}
By symmetry, it suffices to show that Claim \ref{matching} holds for $H_1$. We claim that $H_1$ contains a matching $M$ of size $\m\{s_2,|C_1|\}$. Assume not. Then there exists $R \subseteq S_2\cup C_1$ with $|R|<\m\{s_2,|C_1|\}$ meeting all edges of $H_1$ by K\"{o}nig's theorem. Since $S_1 \cup (C_1 \backslash R)$ is anti-complete to $S_2 \backslash R$, we get a contradiction to the choice of $S_1$ and $S_2$ as $|S_1 \cup (C_1 \backslash R)|+|S_2 \backslash R|>|S_1 \cup S_2|$. Hence, the claim holds.

Since $S_2$ induces a clique in $G$, the matching $M$ is connected. When $|C_1|\leq s_2$, the set $M$ is a dominating matching of $G$ by (b) and (c), which is not possible by Lemma \ref{lem1}. So $|C_1|> s_2$. Equivalently, $M$ is a connected $s_2$-matching. 
\end{proof}

\begin{cla}\label{size}
For each $1\leq i\leq2$, we have $|C_i|>3s_{3-i}$. 
\end{cla}
\begin{proof}
Assume not. By symmetry, we may assume that $|C_1|\leq 3s_2$. Let $M_1$ be a connected $s_2$-matching of $H_1$. 
Claim \ref{matching} implies that such $M_1$ exists. 
Since $|V(G)|-|V(M_1) \cup C_1|\geq 4(t-s_2)-1$, the graph $G\backslash (V(M_1) \cup C_1)$ contains a connected $(t-s_2)$-matching $M_2$ as $G$ is a minimal counterexample to  Conjecture {\rm\ref{conjFGS}}. Moreover, since $M_1$ is a dominating matching of $G\backslash (C_1\backslash V(M_1))$  by (b) and (c), the matching $M_1\cup M_2$ is a connected $t$-matching in $G$, a contradiction. So $|C_1|> 3s_2$.
\end{proof}

By Theorem \ref{17}, we have that $t\geq18$. By (a) and Claim \ref{size}, we have $|C_1|,|C_2|>3$. Moreover, since $s_1+s_2\geq t-3\geq15$, by symmetry and Claim \ref{size} we may assume that $|C_1|\geq 6$. By Ramsey theorem, $G[C_1]$ contains a triangle $T$. Let $H_0$ be the bipartite subgraph of $G$ induced by the edges between $V(T)$ and $C_2$. If $H_0$ does not contain a 3-matching, then by K\"{o}nig's theorem and the fact that $\alpha(G)=2$, there exists a clique $K$ in $G[C_2]$ with $|V(K)|\geq|C_2|-2\geq3s_1-1$. Then $G[S_2\cup V(K)]$ is a clique of size larger than $t-3$, a contradiction to Corollary \ref{omega}. Hence, $H_0$ contains a connected 3-matching $M$ as $T$ is a triangle.

For each $1\leq i\leq2$, since $|C_i \backslash V(M)|> s_{3-i}$ by Claim \ref{size}, following a similar way as the proof of Claim \ref{matching}, $H_i\backslash V(M)$ contains a connected $s_{3-i}$-matching $M_i$. Hence, $M\cup M_1\cup M_2$ is a connected matching of size at least $t$  by (b) and (c), which is a contradiction.
\end{proof}

Using Theorem \ref{thm1} and following a similar way as the proof of Lemma \ref{lem-matching}, we can prove
\begin{cor}\label{}
\textit{Assume that $\cm(G)\le t-1$. For any disjoint subsets $A, B \subseteq V(G)$ with $|A|\leq |B|=t-4$, the bipartite subgraph of $G$ induced by the edges between $A$ and $B$ contains a matching of size $|A|$.}
\end{cor}


\end{document}